\documentclass[12pt]{amsart}
\usepackage{fullpage, amsmath , amsfonts, amscd, epsfig, amssymb}

\theoremstyle{plain}
  \newtheorem{theorem}{Theorem}
  \newtheorem{proposition}[theorem]{Proposition}
  \newtheorem{lemma}[theorem]{Lemma}

\theoremstyle{definition}
  \newtheorem{definition}[theorem]{Definition}
  
\theoremstyle{remark}
  
\theoremstyle{Problem}
  
\newcommand{\field}[1]{\mathbb{#1}}

\newcommand{\newword}[1]{\textbf{\emph{#1}}}
\newcommand{\LE}{%
 \hbox{%
  \vbox{\hrule width 0.35em height 0.04 em}%
  \vbox{\offinterlineskip%
   \hbox{\kern -0.02em\vrule height 0.65em width 0.04em \hspace{0.1em}}%
  }%
 }%
}

\def\I{\mathcal{I}}
\def\J{\mathcal{J}}
\def\F{\mathcal{F}}
\def\M{\mathcal{M}}

\def\R{\field{R}}

\title{Positroids and Schubert Matroids}
\author{Suho Oh}
\address{Department of Mathematics, Massachusetts Institute of Technology, 
        77 Massachusetts Ave, Cambridge, MA 02139}      

\thanks{The author was supported in part by Samsung Scholarship.}

\begin{document}

\begin{abstract}
Postnikov gave a combinatorial description of the cells in a totally-nonnegative Grassmannian. These cells correspond to a special class of matroids called positroids. We prove his conjecture that a positroid is exactly an intersection of cyclically shifted Schubert matroids. This leads to a combinatorial description of positroids that is easily computable. 
\end{abstract}

\maketitle

\section{Introduction}

A \newword{positroid} is a matroid that can be represented by a $k \times n$ matrix with nonnegative maximal minors. The classical theory of total positivity concerns matrices in which all minors are nonnegative, and this subject was extended by Lusztig \cite{L}.

Lusztig introduced the totally nonnegative variety $G \geq 0$ in an arbitrary reductive group $G$ and the totally nonnegative part $(G/P)_{\geq 0}$ of a real flag variety $(G/P)$. He also conjectured that $(G/P)_{\geq 0}$ is made up of cells, and this was proved by Rietsch \cite{R}. 

In this paper, we will restrict our attention to $(Gr_{kn})_{\geq 0}$, the \newword{totally nonnegative Grassmannian}. Then there is a more refined decomposition using matroid strata. Postnikov obtained a relationship between $(Gr_{kn})_{\geq 0}$ and certain planar bicolored graphs, producing a combinatorially explicit cell decomposition of $(Gr_{kn})_{\geq 0}$  \cite{P}. The cells correspond to positroids.

One of the results of \cite{P} is that each cell is an intersection of $(Gr_{kn})_{\geq 0}$ and Schubert cells corresponding to a combinatorial object called the Grassmann necklace. This result implies that each positroid is included in an intersection of cyclically shifted Schubert matroids. We extend this result: each positroid is exactly an intersection of certain cyclically shifted Schubert matroids.

A more detailed formulation of the main result follows.
Let $[n]:=\{1,\cdots,n\}$ and let ${[n] \choose k}$ be the collection of all $k$-element subsets in $[n]$. Fix some $t \in [n]$. We define the ordering $<_t$ on $[n]$ by the total order $t <_t t+1 <_t \cdots <_t n <_t 1 \cdots <_t t-1$. For $I,J \in {[n] \choose k}$, where
$$I=\{i_1, \cdots, i_k \}, i_1 <_t i_2 \cdots <_t i_k$$ and
$$J=\{j_1, \cdots, j_k \}, j_1 <_t j_2 \cdots <_t j_k,$$
we set 
$$I \leq_t J \text{ if and only if } i_1 \leq_t j_1, \cdots, i_k \leq_t j_k.$$ For each $I \in {[n] \choose k}$ and $w \in S_n$, we define the \newword{cyclically shifted Schubert matroid} as
$$ SM^{t}_I := \{J \in {[n] \choose k} | I \leq_t J \}. $$
We will show that a matroid $\M \subseteq {[n] \choose k}$ is a positroid if and only if it can be written as $SM^{1}_{I_1} \cap SM^{2}_{I_2} \cap \cdots \cap SM^{n}_{I_n}$ for a Grassmann necklace $\I = (I_1,\cdots,I_n)$, $I_1, \cdots, I_n \in {[n] \choose k}$. Our proof is purely combinatorial. 

The paper is organized as follows. In section 2, we go over the basics of matroids and the totally nonnegative Grassmannian. In section 3, we review $\LE$-diagrams and $\LE$-graphs. In section 4, we give the proof of our main result. In section 5, we introduce the upper Grassmann necklace. In section 6, we view lattice path matroids as special cases of positroids. In section 7, we define flag positroids.


\medskip

\textbf{Acknowledgment} I would like to thank my adviser, Alexander Postnikov for introducing me to the field and the problem. I would also like to thank Allen Knutson and Lauren Williams for useful discussions. I would also like to thank Anna de Mier and Tom Lenagan for corrections.

\section{Preliminaries and the Main Result}

We would like to guide the readers unfamiliar with basics in this section to \cite{F} and \cite{P}. 

An element in the Grassmannian $Gr_{kn}$ can be understood as a collection of $n$ vectors $v_1, \cdots, v_n \in \R^k$ spanning the space $\R^k$ modulo the simultaneous action of $GL_k$ on the vectors. The vectors $v_i$ are the columns of a $k \times n$-matrix $A$ that represents the element of the Grassmannian. Then an element $V \in Gr_{kn}$ represented by $A$ gives the matroid $\M_V$ whose bases are the $k$-subsets $I \subset [n]$ such that $\Delta_I (A) \not = 0$. Here, $\Delta_I(A)$ denotes the determinant of $A_I$, the $k$ by $k$ submatrix of $A$ with the column set $I$.

Then $Gr_{kn}$ has a subdivision into \newword{matroid strata} $S_{\M}$ labeled by some matroids $\M$:
$$ S_{\M} := \{V \in Gr_{kn} | \M_{V} = \M \}.$$
The elements of the stratum $S_{\M}$ are represented by matrices $A$ such that $\Delta_I (A) \not = 0$ if and only if $I \in \M$.





Let us define the totally nonnegative Grassmannian and its cells.

\begin{definition}[\cite{P}, Definition 3.1] The \newword{totally nonnegative Grassmannian} $Gr^{tnn}_{kn} \subset Gr_{kn}$ is the quotient $Gr^{tnn}_{kn} = GL^{+}_{k} \backslash Mat^{tnn}_{kn}$, where $Mat^{tnn}_{kn}$ is the set of real $k\times n$-matrices $A$ of rank $k$ with nonnegative maximal minors $\Delta_I (A) \geq 0$ and $GL^{+}_k$ is the group of $k \times k$-matrices with positive determinant.
\end{definition}

\begin{definition}[\cite{P}, Definition 3.2] \newword{Totally nonnegative Grassmann cells} $S^{tnn}_{\M}$ in $Gr^{tnn}_{kn}$ are defined as $S^{tnn}_{\M} := S_{\M} \cap Gr^{tnn}_{kn}$. $\M$ is called a \newword{positroid} if the cell $S^{tnn}_{\M}$ is nonempty.
\end{definition}

Note that from above definitions, we get
$$ S^{tnn}_{\M} = \{ GL^{+}_k \bullet A \in Gr^{tnn}_{kn} | \Delta_I (A) >0 \textbf{ for } I \in \M, \Delta_I (A) = 0 \textbf{ for } I \not \in \M  \}. $$

In \cite{P}, Postnikov showed a bijection between each cell and a combinatorial object called the Grassmann necklace.

\begin{definition}[\cite{P}, Definition 16.1]
A \newword{Grassmann necklace} is a sequence $\I = (I_1, \cdots, I_n)$ of subsets $I_r \subseteq [n]$ such that:
\begin{itemize}
\item if $i \in I_i$ then $I_{i+1}= (I_i \setminus \{i \}) \cup \{j\}$ for some $j \in [n]$,
\item if $i \not \in I_i$ then $I_{i+1} = I_i$.  
\end{itemize}
The indices are taken modulo $n$. In particular, we have $|I_1| = \cdots = |I_n|$.
\end{definition}

An example of a Grassmann necklace would be $I_1 = \{1,2,4\}, I_2 = \{2,4,5\},I_3 = \{3,4,5\},I_4 = \{4,5,2\},I_5 = \{5,1,2\}$. Two of the results in \cite{P} are the following: 

\begin{lemma}[\cite{P}, Lemma 16.3]
\label{lem:P}
For a matroid $\M \subseteq {[n] \choose k}$ of rank $k$ on the set $[n]$, let $\I_{\M} = (I_1,\cdots,I_n)$ be the sequence of subsets such that $I_i$ is the minimal member of $\M$ with respect to $\leq_i$. Then $\I_{\M}$ is a Grassmann necklace.
\end{lemma}

\begin{theorem}[\cite{P}, Theorem 17.2] 
\label{thm:P}
Let $S^{tnn}_{\M}$ be a nonnegative Grassmann cell, and let $\I_{\M} = (I_1, \cdots, I_n)$ be the Grassmann necklace corresponding to $\M$. Then 
$$S^{tnn}_{\M} = \bigcap_{i=1}^{n} \Omega_{I_i}^{i} \cap {Gr}^{tnn}_{kn},$$
where $\Omega_{I_i}^{i}$ is the cyclically shifted Schubert cell, which is the set of elements $V \in Gr_{kn}$ such that $I_i$ is the lexicographically minimal base of $M_V$ with respect to ordering $<_i$ on $[n]$.
\end{theorem}

These results imply that bases of a positroid are included in an intersection of cyclically shifted Schubert matroids. But it does not imply that they are equal. Postnikov therefore conjectured that each positroid is exactly the intersection of cyclically shifted Schubert matroids. This is what we are going to prove in our paper:

\begin{theorem} 
\label{thm:mainO}
$\M$ is a positroid if and only if for some Grassmann necklace $(I_1, \cdots, I_n)$,
$$ \M = \bigcap_{i=1}^{n} SM^i_{I_i}. $$
In other words, $\M$ is a positroid if and only if the following holds : $H \in \M$ if and only if $H \geq_t I_t$ for all $t\in [n]$.
\end{theorem}

\section{Le-diagrams and Le-graphs}

In \cite{P}, Postnikov showed a bijection between positroids and combinatorial objects called $\LE$-diagrams. 

\begin{definition}
Fix a partition $\lambda$ that fits inside the rectangle $(n-k)^k$. The boundary of the Young diagram of $\lambda$ gives the lattice path of length $n$ from the upper right corner to the lower left corner of the rectangle $(n-k)^k$. Let's denote this path as the \newword{boundary path}. Label each edge in the path by $1,\cdots,n$ as we go downwards and to the left. Define $I(\lambda)$ as the set of labels of $k$ vertical steps in the path. 

Each column and row corresponds to exactly one labeled edge. Let's index the columns and rows with those labels. We will say that a box is at $(i,j)$ if it is on row $i$ and column $j$. A \newword{filling} of $\lambda$ is a diagram of $\lambda$ where each box is either empty or filled with a dot. 
\end{definition}

Given a filling $L$ of shape $\lambda$, we define the sets $box(L),filled(L)$ as:
$$box(L) := \{ (i,j) | \text{ there is a box at } (i,j) \text{ in } L \},$$
$$filled(L) := \{ (i,j) | \text{ there is a box filled with a dot at } (i,j) \text{ in } L \}.$$


\begin{definition}[\cite{P}, Definition 6.1]
 For a partition $\lambda$, let us define a \newword{$\LE$-diagram} $L$ of shape $\lambda$ as a filling of boxes of the Young diagram of shape $\lambda$ such that, for any three boxes indexed $(i,j),(i',j),(i,j')$, where $i'<i$ and $j'>j$, if boxes on position $(i',j)$ and $(i,j')$ are filled, then the box on $(i,j)$ is also filled. This property is called the $\LE$-property. We will say that a $\LE$-diagram is \newword{full} if every box is filled. 
\end{definition}

\begin{figure}
	\centering
	  \includegraphics[width=0.3\textwidth]{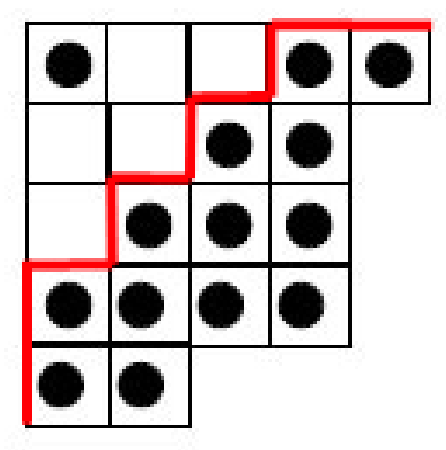}
	  \includegraphics[width=0.3\textwidth]{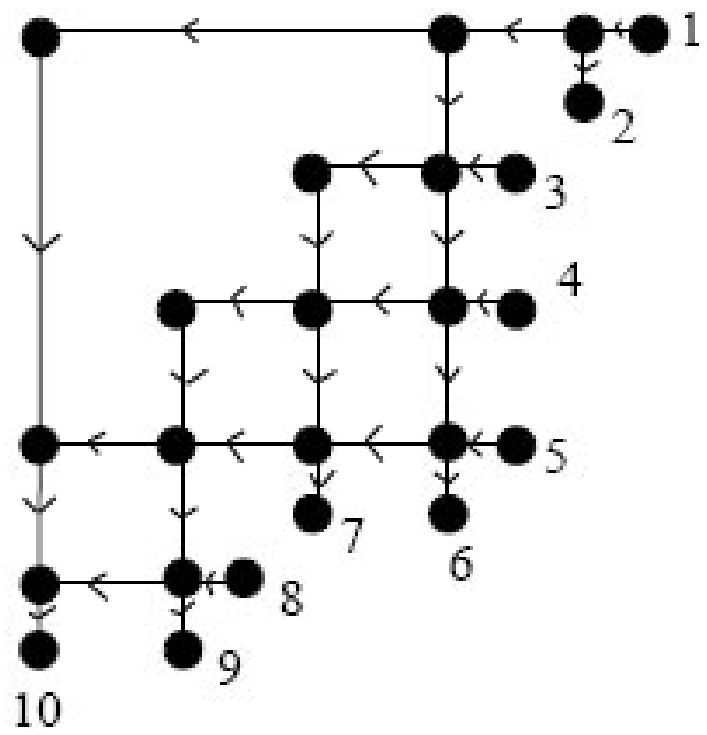}
	\caption{Example of a $\LE$-diagram and a $\LE$-graph}
	\label{fig:leh}
\end{figure}

Fix a $\LE$-diagram $L$ of shape $\lambda$. For each $(i,j) \in box(L)$, we define the \newword{NW-region} of it as:
$$NW_{(i,j)} := \{(i',j')|i'<i, j'>j\}.$$
 There is a unique dot $(i',j')$ that minimizes $i-i'$ and $j'-j$ at the same time, due to the $\LE$-property. We will say that $(i',j')$ covers $(i,j)$ and write this as $(i',j') \triangleleft (i,j)$.
 

\begin{definition}[\cite{P}]
A \newword{$\LE$-graph} is obtained from a $\LE$-diagram in the following way. Place a vertex at the middle of each step in the boundary path of the diagram and mark these vertices by $1,2,\cdots,n$. We will call these vertices the \newword{boundary vertices}. Now for each dot inside the $\LE$-diagram, draw a horizontal line to its right, and vertical line to its bottom until it reaches the boundary of the diagram. Then orient all vertical edges downward and horizontal edges to the left.
\end{definition}

$\LE$-graphs were also used to study TP-basis of positroids in \cite{Talaska}. The source set of the $\LE$-graph is given by $I(\lambda)$ and the sink set is given by $[n] \setminus I(\lambda)$. 

\begin{definition}
A \newword{path} in a $\LE$-graph is a directed path that starts at some boundary vertex and ends at some boundary vertex. Given a path $p$, we denote its starting point and end point by $p^s$ and $p^e$. A \newword{VD-family} is a family of paths where no pair of paths share a vertex. 
\end{definition}

A dot at $(i,j)$ is a \newword{NW-corner} of a path $p$, if $p$ changes direction at $(i,j)$. For each $(i,j) \in filled(L)$, there is a path that starts at a boundary vertex $i$, ends at a boundary vertex $j$ and has the dot at $(i,j)$ as a NW-corner. We call such path as a \newword{hook path of $(i,j)$}.


Given a VD-family of  paths $\{p_1,\cdots,p_t\}$, we say that this family represents $J = I(\lambda) \setminus \{p_1^s,\cdots,p_t^s\} \cup \{p_1^e,\cdots,p_t^e\}$. Empty family is also a VD-family. The following proposition follows as a corollary from (\cite{P}, Theorem 6.5).

\begin{proposition}[\cite{P}]
Fix a positroid $\M_{L}$ that corresponds to a $\LE$-diagram $L$. Then $J \in \M_{L}$ if and only if $J$ is represented by a VD-family in the $\LE$-graph of $L$. 
\end{proposition}

Each $\LE$-diagram corresponds to a positroid, and hence a Grassmann necklace. Given a $\LE$-diagram $L$, let's try to find out its corresponding Grassmann necklace $\I = (I_1,\cdots,I_n)$ directly from the diagram. It is obvious that $I_1 = I(\lambda)$. 

For each $(x,y) \in box(L)$, we can get a maximal chain $(x_t,y_t) \triangleleft \cdots \triangleleft (x_1,y_1)$ such that $(x_1,y_1)$ is the unique dot in $\{(i,j)| i \leq x, j \geq y\}$ that minimizes $x-i$ and $j-y$ at the same time. We will call this the \newword{chain rooted at $(x,y)$}. Then the collection of hook paths at $(x_r,y_r)$ for $1 \leq r \leq t$ is a VD-family.  So we get $J_{(x,y)} := I_{\lambda} \setminus \{x_1,\cdots,x_t\} \cup \{y_1,\cdots,y_t\} \in \M_L$. In Figure~\ref{fig:leh}, chain rooted at $(5,9)$ is given by $(1,10) \triangleleft (5,9)$. A chain rooted at $(3,9)$ is given by $(1,10)$.

\begin{proposition}
Fix $j \in [n] \setminus \{1\}$ and a $\LE$-diagram $L$ of shape $\lambda$ and let $\I = (I_1,\cdots,I_n)$ be the Grassmann necklace of $\M_{L}$. If $j \not \in I(\lambda)$, let $(x,y)$ be the box adjacent to a path labeled $j$ in the boundary path of $\lambda$. If $j \in I(\lambda), j \not = 1$, let $(x,y)$ be the box right above such box. Then $I_j = J_{(x,y)}$.
\end{proposition}
\begin{proof}
Let $\F$ be a VD-family that represents $I_j$. Then $\F$ only contains paths that satisfy $p^s < j \leq p^e$. Because if not, then $\F \setminus \{p\}$ represents $J$ such that $J <_j I_j$. So any path $p \in \F$ has to pass through a dot in the region $\{(i,j)| i \leq x, j \geq y\}$.


Let the chain rooted at $(x,y)$ be $(i_t,j_t) \triangleleft \cdots \triangleleft (i_1,j_1)$. For each $1 \leq r \leq t$, denote the hook path at $(i_r,j_r)$ by $p_r$. Then $\F$ must contain $p_1$. If not, then $I_j \not \leq_j I(\lambda) \setminus \{i_1\} \cup \{j_1\}$ because ${p}^s \leq i_1,{p}^e \geq j_1$ for all $p \in \F$. If $p_1,\cdots,p_r \in \F$, then we also have $p_{r+1} \in \F$ because if not, we get $I_j \not \leq_j I(\lambda) \setminus \{i_1,\cdots,i_{r+1}\} \cup \{j_1,\cdots,j_{r+1}\}$ due to the fact that for any path $p \in \F \setminus \{p_1,\cdots,p_r\}$, we have ${p}^s \leq i_{r+1}$ and ${p}^e \geq j_{r+1}$. As a result, we get $\F = \{p_1,\cdots,p_t\}$ and $I_j = J_{(x,y)}$.

\end{proof}

Let's look at an example. In the $\LE$-diagram of Figure~\ref{fig:leh}, $I_4$ is given by $J_{(3,6)}$. Chain rooted at $(3,6)$ is given by $(1,10) \triangleleft (3,6) $. So $I_4 = I_1 \setminus \{1,3\} \cup \{10,6\}  = \{4,5,6,8,10\}$. $I_9$ is given by $J_{(8,9)}$. Chain rooted at $(8,9)$ is given by $(5,10) \triangleleft (8,9)$. So $I_9 = I_1 \setminus \{5,8\} \cup \{9,10\} =  \{1,3,4,9,10\}$.


\section{Proof of the main theorem}

In this section, we will prove the main theorem by showing that for each Grassmann necklace $\I = (I_1,\cdots,I_n)$, we have $\bigcap_{i=1}^{n} SM^i_{I_i} \subseteq \M_{\I}$. To do this, we need to show that each $J \in \bigcap_{i=1}^{n} SM^i_{I_i}$ can be expressed as VD-family inside the $\LE$-graph of $\M_{\I}$. In order to accomplish this, we will start from a full-$\LE$-diagram and use induction by increasing the number of empty boxes.


\begin{lemma}
Let $L$ be a full $\LE$-diagram of shape $\lambda$. Then $\M_{L} = SM_{I(\lambda)}$.
\end{lemma}

\begin{proof}
We need to show that for all $J \in SM_{I(\lambda)}$, we have $J \in \M_{L}$. Due to the definition of $\leq_1$, there is a unique bijection $\phi : I(\lambda) \setminus J \rightarrow J \setminus I(\lambda)$ such that for any $a,b \in I(\lambda) \setminus J$, the two intervals $[a,\phi(a)]$ and $[b,\phi(b)]$ do not cross, meaning that they are either disjoint or nested. For each $a \in I(\lambda) \setminus J$, we associate a hook path at $(a,\phi(a))$. Then we get a VD-family representing $J$.
\end{proof}

 Given any $\LE$-diagram $L_{\I}$ with associated Grassmann necklace $\I = (I_1,\cdots,I_n)$, we want to add a dot, to obtain a new $\LE$-diagram $L_{\I'}$ such that for some $\alpha \in [n]$, we have $|I_\alpha \setminus {I_\alpha}'|=1$ and ${I_i}' = I_i$ whenever $i \not = \alpha$.

The \newword{boundary strip} of $L_{\I}$ is the set of boxes where it shares at least one vertex with the boundary path. Let's first assume that there exists an empty box in the boundary strip of $L_{\I}$. Consider an empty box in the strip such that there is no empty box to its right or bottom. Then adding a dot to this box will change exactly one element of the Grassmann necklace. So we only need to consider the case when all the boxes of the boundary strip are filled. We define the \newword{middle path} of $L_\I$ to be a lattice path inside the diagram such that:
\begin{enumerate}
\item all boxes between the middle path and the boundary path are filled with dots,
\item the corner boxes of the upper region is empty. \newword{Upper region} is the diagram obtained by looking at the boxes above or left of the middle path. A box is a \newword{corner box} of a diagram if there is are no boxes to its right and below.
\end{enumerate}
 Then putting a dot into any corner box of the upper region will work, since the newly added dot will affect only one element of the Grassmann necklace. Example of a middle path is given as a red line in Figure~\ref{fig:leh}.
 
\begin{proposition} 
Given any Grassmann necklace $\I=(I_1, \cdots, I_n)$, we have $\M_{\I}=\bigcap_{i=1}^{n} SM^i_{I_i}.$
\end{proposition}

\begin{proof}
We will prove the proposition by induction on $m$, the number of empty boxes inside the $\LE$-diagram $L_{\I}$ of $\M_{\I}$. When $m=0$, this is the full $\LE$-diagram case. So assume for the sake of induction that we know the result for $\LE$-diagrams having $<m$ empty boxes. 

Use the construction above to obtain $L_{\I'}$, where $\I'=({I_1}',\cdots,{I_n}')$ and there exists $\alpha \in [n]$ such that ${I_i}' = I_i$ for all $i \not = \alpha$ and $|I_\alpha \setminus {I_\alpha}'|=1$. Induction hypothesis tells us that $\M_{\I'} =  \bigcap_{i=1}^{n} SM^i_{{I_i}'}$. It is enough to show $\M_{\I'} \setminus \M_{\I} \subset {SM^{\alpha}_{I_\alpha'}} \setminus {SM^{\alpha}_{I_\alpha}}$. 

Let $(w_{q+r},z_{q+r}),\cdots \triangleleft (w_q,z_q) \triangleleft \cdots \triangleleft (w_{1}, z_{1})$ be the chain representing $I_{\alpha}'$ in $L_{\I'}$, such that $(w_q,z_q)$ is the newly added dot going from $L_{\I}$ to $L_{\I'}$. We have $(w_a,z_b) \in filled(L_{\I'})$ for $1 \leq a,b \leq q$. Any VD-family $\F_J$ representing some $J \in \M_{\I'} \setminus \M_{\I}$ should contain a path in which $(w_q,z_q)$ is a NW-corner.

In $\F_J$, denote the path going through $(w_q,z_q)$ by $p_q$. If there is no path in $\F_J$ that passes $(w_{q-1},z_{q-1})$, we can perturb the path $p_q$ to go through the points $(w_q,z_{q-1})$, $(w_{q-1},z_{q-1})$, $(w_{q-1},z_q)$ instead of going through $(w_q,z_q)$. So there must be a path $p_{q-1} \in \F_J$ that passes $(w_{q-1},z_{q-1})$. Since $(w_q,z_q)$ is a NW-corner of $p_q$, $(w_{q-1},z_{q-1})$ is also a NW-corner of $p_{q-1}$. Repeating this argument, we get $p_q,\cdots,p_1 \in \F_J$ each having  $(w_q,z_q),\cdots,(w_1,z_1)$ as a NW-corner.

Let $(x_t,y_t) \triangleleft \cdots \triangleleft (x_1,y_1)$ be the chain rooted at $(w_q,z_q)$ in $L_{\I}$. Then 
$$(x_t,y_t) \triangleleft \cdots \triangleleft (x_1,y_1) \triangleleft (w_{q-1},z_{q-1}) \triangleleft \cdots \triangleleft (w_{1}, z_{1})$$ represents $I_j$ in $\L_{\I}$. We have $t \geq r$ due to the $\LE$-property. We want to show that $J \not \geq_{\alpha} I_\alpha$.


If $p_q^e <_\alpha y_1$ or $p_q^s >_\alpha x_1$, then we have $J \not \geq_\alpha I_\alpha$ and we are done. So let's assume $p_q^e \geq_\alpha y_1$ and $p_q^s \leq_\alpha x_1$. If there is no path going through $(x_1,y_1)$ in $\F_J$, the path $p_q$ can be slightly changed so it goes through $(x_1,y_1)$ and this path cannot have $(w_q,z_q)$ as its NW-corner. So there must be a path $p_{q+1}$ in $\F_J$ that passes through $(x_1,y_1)$. 

Due to similar reasons, we only need to consider the case when $p_{q+1}^e \geq_\alpha y_2$ and $p_{q+1}^s \leq_\alpha x_2$. If there is no path going through $(x_2,y_2)$ in $\F_J$, the path $p_{q+1}$ can be slightly changed so it goes through $(x_2,y_2)$. This path cannot pass $(x_1,y_1)$, since we have $x_2 <_\alpha x_1$ and $y_2 >_\alpha y_1$. So there must be a path $p_{q+2} \in \F_J$ that passes through $(x_2,y_2)$. Repeating this argument, we get $p_{q+1},\cdots,p_{q+t} \in \F_J$. Then $\{p_{q+t}^e,\cdots,p_1^e\} \subset J$ tells us that $J \not \geq_{\alpha} I_{\alpha}$. (The reason we do this separately from the previous paragraph is because one of $y_1= z_q$ and $x_1 = w_q$ might be true.)


So we have shown $\M_{\I'} \setminus \M_{\I} \subset {SM}^{\alpha}_{I_\alpha'} \setminus {SM^{\alpha}_{I_\alpha}}$, and we are finished.

\end{proof}

Let's look at an example on using the main theorem. Let $\M$ be a positroid such that its Grassmann necklace is given by:
$$I_1 = \{1,2,4\},I_2 = \{2,4,5\},I_3 = \{3,4,5\},I_4 = \{4,5,2\},I_5 = \{5,1,2\}.$$
Our main theorem tells us that:
$$\M = \{H | H \geq_1 I_1, H\geq_2 I_2, \cdots, H\geq_5 I_5 \}$$
$$=\{ \{1,2,4\} , \{1,2,5\}, \{1,3,4\}, \{1,3,5\}, \{2,4,5\}, \{3,4,5\} \}.$$




\section{Decorated permutations and the Upper Grassmann necklace}

In this section we will show that a positroid is also an intersection of cyclically shifted dual Schubert matroids.

\begin{definition}[\cite{P}, Definition 13.3]
 A decorated permutation $\pi^{:} = (\pi, col)$ is a permutation $\pi \in S_n$ together with a coloring function $col$ from the set of fixed points $\{i | \pi(i) = i\}$ to $\{1,-1\}$. That is, a decorated permutation is a permutation with fixed points colored in two colors.
\end{definition}

It is easy to see the bijection between necklaces and decorated permutations. To go from a Grassmann necklace $\I$ to a decorated permutation $\pi^{:}=(\pi,col)$,

\begin{itemize}
 \item if $I_{i+1} = (I_i \backslash \{i\}) \cup \{j\}$, $j \not = i$, then $\pi(i)=j$,
 \item if $I_{i+1} = I_i$ and $i \not \in I_i$ then $\pi(i)=i, col(i)=1$,
 \item if $I_{i+1} = I_i$ and $i \in I_i$ then $\pi(i)=i, col(i)=-1$.
\end{itemize}

To go from a decorated permutation $\pi^{:}=(\pi,col)$ to a Grassmann necklace $\I$,
$$I_i = \{ j \in [n] | j<_i \pi^{-1}(j) \textbf{ or } (\pi(j)=j \textbf{ and } col(j)=-1) \}.$$

Let's look at an example. For decorated permutation $\pi^{:}$ with $\pi = 81425736$ and $col(5)=1$, we get $I_1 = \{1,2,3,6\},I_2 = \{2,3,6,8\},I_3 = \{3,6,8,1\},I_4 = \{4,6,8,1\},I_5 = \{6,8,1,2\}, I_6 = \{6,8,1,2\}, I_7 = \{7,8,1,2\}, I_8 = \{8,1,2,3\}$.

\begin{definition}
For $I=(i_1, \cdots, i_k) \in {[n] \choose k}$, the \newword{cyclically shifted dual Schubert matroid} $\tilde{SM}^{i}_I$ consists of bases $H=(j_1, \cdots, j_k)$ such that $I \geq_i H$.
\end{definition}

Fix a decorated permutation $\pi^{:}=(\pi,col)$. Let $\I_{\pi^{:}} = (I_1,\cdots,I_n)$ be the corresponding Grassmann necklace and $\M_{\pi^{:}}$ the corresponding positroid. 

\begin{lemma}
\label{lem:maxM}
For any $H \in \M_{\pi^{:}}$, we have $H \leq_i \pi^{-1} (I_i)$ for all $i \in [n]$.
\end{lemma}

\begin{proof}
We may assume that $\pi$ has no fixed point since they correspond to loops or coloops of $\M_{\pi^{:}}$. We will prove only for $i=1$ since the proof is similar in other cases. Denote $I_1 = \{i_1,\cdots,i_k\}$ where $i_1,\cdots,i_k$ are labeled in a way that satisfies $\pi^{-1}(i_1) < \cdots < \pi^{-1}(i_k)$.


Denote elements of $H$ by $h_1 < \cdots < h_k$. Let $j$ be the biggest element of $[k]$ such that: 
\begin{enumerate}
\item $h_t \leq \pi^{-1}(i_t)$ for all $t \in (j,k]$ and
\item $h_j > \pi^{-1}(i_j)$.
\end{enumerate}
Since $h_i \in (\pi^{-1}(i_j), \pi^{-1}(i_{j+1})]$, we have $\{i_1,\cdots,i_j\} \subset I_{h_j}$. We get $|H \cap [1,h_j)| < |I_{h_j} \cap [1,h_j)|$, but this contradicts $H \geq_{h_j} I_{h_j}$. Hence there cannot be a $j \in [k]$ such that $h_j > \pi^{-1}(h_j)$. This tells us that $H \leq \{\pi^{-1}(j_1),\cdots,\pi^{-1}(j_k) \}$.


\end{proof}

The collection $(J_1:=\pi^{-1}(I_1),\cdots,J_n:=\pi^{-1}(I_n))$ forms a necklace in the sense that $J_{i+1} = J_i \setminus \{\pi^{-1}(i)\} \cup \{i\}$ except for $i$ such that $\pi(i)=i$. We will call this the \newword{upper Grassmann necklace} of $\pi$. 

To go from a decorated permutation $\pi^{:}=(\pi,col)$ to an upper Grassmann necklace $\J$,
$$J_r = \{ i \in [n] | \pi(i) <_r i \textbf{ or } (\pi(i)=i \textbf{ and } col(i)=-1) \}.$$

Define $\tilde{\M}_{\pi^{:}}$ as:
$$ \tilde{\M}_{\pi^{:}} = \bigcap_{i=1}^{n} \tilde{SM}_{J_i}^{i}.$$

Then Lemma~\ref{lem:maxM} tells us that $\M_{\pi^{:}} \subseteq \tilde{\M}_{\pi^{:}}$.
The proof of the following lemma is similar to Lemma~\ref{lem:maxM}.

\begin{lemma}
For any $H \in \tilde{\M}_{\pi^{:}}$, we have $H \geq_i \pi(J_i)=I_i$ for all $i \in [n]$.
\end{lemma}

So we obtain the following result:

\begin{theorem}
\label{thm:duality}
Pick a decorated permutation $\pi^{:}=(\pi,col)$. Let $\I=(I_1,\cdots,I_n)$ and $\J=(J_1,\cdots,J_n)$ be the corresponding Grassmann necklace and the upper Grassmann necklace. Then $J_i = \pi^{-1}(I_i)$ for all $i \in [n]$. We also have the equality:
$$ \bigcap_{i=1}^{n} SM^i_{I_i} = \bigcap_{i=1}^{n} \tilde{SM}_{J_i}^{i}.$$
\end{theorem}

\section{Lattice Path Matroids}
Lattice path matroids were defined in \cite{BMN}. These are simple cases of positroids. In this section we will show a way to get the decorated permutation of a lattice path matroid.

\begin{definition}
Pick $I,J \in {[n] \choose k}$ such that $I \leq J$. The \newword{lattice path matroid} is defined as: 
$$ LP_{I,J} := \{H| H \in {[n] \choose k}, I \leq H \leq J \} = SM_I \cap \tilde{SM}_J $$
\end{definition}

Since $I,J$ corresponds to two lattice paths in a $(n-k)$-by-$k$ grid, $LP_{I,J}$ expresses all the lattice paths between them. 

\begin{lemma}
\label{lem:pathposi}
A lattice path matroid is a positroid.
\end{lemma}

\begin{proof}
Pick $I=\{a_1 < \cdots < a_k\}, J=\{b_1 < \cdots < b_k\}$ such that $I \leq J$. Let's construct a $k$-by-$n$ matrix such that $\Delta_H = 0$ for all $H \in {[n] \choose k} \setminus LP_{I,J}$ and $\Delta_H >0$ for all $H \in LP_{I,J}$.

Let $V=(v_{ij})_{i,j=1,1}^{k,n}$ be a $k$-by-$n$ Vandermonde matrix. Set $v_{ij}=0$ for all $j \not \in [a_i,b_i]$. So $V$ would look like:

$$ v_{ij} = \{ \begin{array}{ll}
         {x_i}^{j-1} & \mbox{if $a_i \leq j \leq b_i,$}\\
        0 & \mbox{otherwise.} \end{array}$$ 

Assign values to variables $x_1,\cdots,x_k$ such that $x_1 >1$ and $x_{i+1} = x_i^{k^2}$ for all $i \in [k-1]$. Let's denote $V_{[1..i],[c_1,\cdots,c_i]}$ as a submatrix of $V$ by taking rows from $1$ to $i$ and columns $c_1,\cdots,c_i$.  We have $\Delta_{H}>0$ if and only if $V_{[1..k],H}$ has nonzero diagonal entries, which happens if and only if $H \in LP_{I,J}$. 


\end{proof}

Let's try to find $\pi^{:}$ that corresponds to $LP_{I,J}$. Denote $I=\{i_1 < \cdots <i_k\}$ and $J=\{j_1 < \cdots < j_k\}$. If $i_t=j_t$ for some $t \in [k]$, this is a coloop in the matroid. The permutation $\pi$ we are trying to find should satisfy:
\begin{itemize}
\item $I = \{ i \in [n] | i \leq \pi^{-1}(i) \}$, 
\item $J = \{ i \in [n] | \pi(i) \leq i) \}$ and
\item $\pi(J)=I$.
\end{itemize}

If $\pi$ satisfies the above conditions, then $\M_{\pi^{:}}$ is contained in $LP_{I,J}$. So $LP_{I,J}$ is the biggest positroid under inclusion inside the collection satisfying the above property. The following lemma is an immediate corollary of (\cite{P}, Theorem~17.8).

\begin{lemma}
If we have $a <_i b <_i \pi(a) <_i \pi(b) <_i a$, then $\M_{\pi^{:}} \subset \M_{\mu^{:}}$, where $\mu$ is obtained from $\pi$ by switching $\pi(a)$ and $\pi(b)$ (i.e. $\mu(a) = \pi(b), \mu(b) = \pi(a)$).
\end{lemma}






Combining this with Lemma~\ref{lem:pathposi}, we get the following result:

\begin{theorem}
Choose any $I=\{i_1 < \cdots < i_k\}$ and $J=\{j_1 < \cdots < j_k\} \in \ {[n] \choose k}$ such that $I \leq J$. Denote $[n] \setminus J = \{d_1 < \cdots <d_{n-k}\}$ and $[n] \setminus I = \{c_1 < \cdots< c_{n-k}\}$. Then $LP_{I,J}$ is a positroid and its decorated permutation $\pi^{:}=(\pi,col)$ is given by:
$$\mbox{$\pi(j_r) = i_r$ for all $r \in [k]$},$$
$$\mbox{$\pi(d_r) = c_r$ for all $r \in [n-k]$},$$
$$ \mbox{If $\pi(t)=t$ then $col(t)$} = \{ \begin{array}{ll}
        -1 & \mbox{if $t \in J$,}\\
        1 & \mbox{otherwise.} \end{array} $$ 
        
\end{theorem}

\section{Further Remark}






Positroids correspond to the matroid strata of the nonnegative part of the Grassmannian. Flag matroids correspond to the matroid strata of a flag variety. 

\begin{definition}
A \newword{flag} $F$ is a strictly increasing sequence
$$ F^1 \subset F^2 \subset \cdots \subset F^m $$
of finite sets. Denote by $k_i$ the cardinality of the set $F^i$. We write $F=(F^1,\cdots,F^m)$. The set $F^i$ is called the $i$-th constituent of $F$.

\end{definition}

\begin{theorem}[\cite{BGW}]
A collection $\F$ of flags of rank $(k_1,\cdots,k_m)$ is a \newword{flag matroid} if and only if:
\begin{enumerate}
\item For all $i \in [m]$, $M_i$ the collection of $F^i$'s for each $F \in \F$ form a matroid.
\item For every $w \in S_n$, the $\leq_w$-minimal bases of each $M_i$ form a flag. If this holds, we say that $M_i$'s are concordant.
\item Every flag
$$B_1 \subset \cdots \subset B_m$$
such that $B_i$ is a basis of $M_i$ for $i=1,\cdots,m$ belongs to $\F$.
\end{enumerate}
\end{theorem}

\begin{definition}
A \newword{flag positroid} is a flag matroid in which all constituents are positroids.
\end{definition}

It would be interesting to check what is the necessary condition for two decorated permutations, so that the corresponding positroids are concordant.


\begin{thebibliography}{HLSS}

\bibitem[BGW]{BGW} A.~Borovik, I.~Gelfand, N.~White. \textit{Coxeter Matroids} Birkhauser, Boston, 2003.

\textit{Oriented Matroids}. Encyclopedia of Mathematics and Its Applications,
vol. 46. Cambridge University Press, Cambridge, 1993.

\bibitem[BMN]{BMN} J.~Bonin, A.~de Mier, M.~Noy. Lattice path matroids: enumerative aspects and Tutte polynomials,\textit{ J. Combin. Theory Ser. A} 104 (2003) 63-94.

\bibitem[F]{F} W.~Fulton. \textit{Young Tableaux. With Applications to Representation Theory and Geometry} New York: Cambridge University Press, 1997


\bibitem[L]{L} G.~Lusztig, Introduction to total positivity, in \textit{Positivity in Lie theory: open problems, ed.}
J. Hilgert, J.D. Lawson, K.H. Neeb, E.B. Vinberg, de Gruyter Berlin, 1998, 133-145

\bibitem[P]{P} A.~Postnikov: Total positivity, Grassmannians, and networks, 
arXiv: math/ 0609764v1 [math.CO].

\bibitem[T]{Talaska} K.~Talaska: Combinatorial formulas for Le-coordinates on a totally nonnegative Grassmannian, 
to appear in \textit{J. Combin. Theory Ser. A}.

\bibitem[R]{R} K.~Rietsch, Total positivity and real flag varieties, Ph.D. Dissertation, MIT, 1998.


\end{thebibliography}
\end{document}